\documentclass[11pt, epsfig]{article}
\usepackage{epsfig, amsmath, amssymb, amsthm, times}
\usepackage{epsfig, amsmath, amssymb, amsthm, times, authblk}
\usepackage[usenames,dvipsnames]{xcolor}

\newcommand{\bel}{\begin{equation}\label}
\newcommand{\ee}{\end{equation}}
      \newtheorem{theorem}{Theorem}[section]

       \newtheorem{lemma}[theorem]{Lemma}
       
\theoremstyle{definition}

\def\eps{\varepsilon}

\def\N{{\mathbb N}}

\def\<{\langle}
\def\>{\rangle}
\def\P{\mathbb P}

\def\E{\mathbb E}

\def\eps{\epsilon}

\def\0{\underline 0}

\def\1{\underline 1}

\def\var{\mathbb{V}\mathrm{ar}}
\def\cov{\mathbb{C}\mathrm{ov}}

\def\Bin{{\operatorname{Bin}}}
\def\var{{{\mathbb V}\mbox{ar}}}
\def\cov{{{\mathbb C}\mbox{ov}}}
\def\cF{\mathcal F}
\def\be*{\begin{equation*}}
\def\ee*{\end{equation*}}
\def\bar*{\begin{eqnarray*}}
\def\ear*{\end{eqnarray*}}
\def\be{\begin{equation}}
\def\ee{\end{equation}}

\def\pxn{p^{}_{X_n}}
\def\qxn{q^{}_{X_n}}
\def\pyn{p^{}_{Y_n}}
\title{Asymptotics of the overflow in urn models\footnote{This material is based upon work supported by and while serving at the National Science Foundation. Any opinion, findings, and conclusions or recommendations expressed in this material are
those of the authors and do not necessarily reflect the views of the National Science Foundation.}\ \footnote{Part of the research by the last two authors was carried out while they
visited the Center for Mathematical Modeling at the University of Chile.
They would like to thank the first author for arranging the visits and his
hospitality, and the CMM for a generous support.}}
\author[1]{Raul Gouet\thanks{Supported by grants PIA AFB-170001 and Fondecyt 1161319}}
\author[2]{Pawe{\l} Hitczenko\thanks {On leave from Drexel University}}
\author[3]{Jacek Weso{\l}owski\thanks{Supported by the grant 2016/21/B/ST1/00005 of National Science Centre, Poland}}
\affil[1]{Departamento de Ingenieria Matem\'atica and CMM  (UMI 2807, CNRS), Universidad de Chile}
\affil[2]{Division of Mathematical Sciences, National Science Foundation}
\affil[3]{Faculty of Mathematics and Information Science, Warsaw University of Technology}
\begin{document}
\maketitle
\begin{abstract}
	Consider a number, finite or not, of urns each with fixed capacity $r$ and balls randomly distributed among them. An overflow is the number of balls that are assigned to urns that already contain $r$ balls.  When $r=1$, using analytic methods,  Hwang and Janson gave conditions under which the  overflow (which in this case is just the number of balls landing in non--empty urns)  has an asymptotically Poisson distribution as the number of balls grows to infinity.  Our aim here is to  systematically study the asymptotics of the overflow in general situation, i.~e. for arbitrary $r$. In particular,    we
	provide sufficient conditions for both Poissonian and normal asymptotics for general $r$, thus extending Hwang--Janson's work. Our approach  relies  on purely probabilistic methods.\\
	
\end{abstract}
\section{Introduction}
Urn models are one of the fundamental objects in classical probability theory and they have been studied for a long time in various degrees of generality. We refer the reader to classical sources \cite{JK,KSC,KB,M} for a complete account of the theory and discussions of different models, and to e.~g. \cite{GHP, HJ, BHLRW} for some of the more recent developments.   Perhaps the most heavily
studied characteristic is the number of occupied urns after $n$ balls have been thrown in. One reason for this is that it is often interpreted as a measure of diversity of a given population. Actually,  more refined characteristics, e.~g. the number of urns containing the prescribed number of balls, have been subsequently studied for various urn models. In diversity analysis, the number $M_k$ of urns with exactly $k$ balls, is called abundance count of order $k$. In particular,  the popular estimator of species richness, called Chao estimator, is  based on $M_1$ and $M_2$ (with a more sophisticated version using also $M_3$ and $M_4$) - see e.~g. \cite{CC}.  In \cite{HJ} the authors
used analytical methods based on Poissonization and de--Poissonization to prove that the number of empty urns is asymptotically normal as long as its variance grows to infinity (this is clearly the minimal requirement). As a by--product of their method they
established the Poissonian asymptotics of the number of balls that fall into non--empty urns when the variance is finite and under additional assumptions on the distribution among boxes. We mention in passing that the number of balls falling into non--empty urns is sometimes called the number of collisions.
Under the uniformity assumption for the distribution of balls it has been used, for example,  for testing random number generators (see \cite[\S 3.3.2~I]{K2} for more details). We refer also to \cite{AGK} and references therein for another illustration of how this concept is used, e.g. in cryptology.

Our main aim here is to extend  the result of Hwang and Janson by considering the number of balls falling into urns containing at least $r$ balls (thus, their result corresponds to $r=1$). Relying on purely probabilistic methods we  provide sufficient conditions for both Poissonian and normal asymptotics for the number of balls falling into such urns.

One way to formulate the problem is as follows. There is a collection (possibly infinite) of  distinct containers in which balls are to be inserted. All containers have the same finite capacity. Each arriving ball is to be placed in one of the containers, randomly and independently of other balls. However, if the container selected for a given ball is already full, the ball lands in the overflow basket. We are interested in the number of balls in that basket when more and more balls appear. The notion of the overflow is not entirely new and has appeared, for example,  in the context of collision resolution  for hashing algorithms, see a
discussion in   section: ``External searching''  in \cite[\S 6.4]{K3}. We also refer to subsequent work \cite{R, Mo} for  the computation of the probability that there is no overflow (under the uniformity assumption), and to \cite{DNW} which, in part, concerns the estimation of the probability of unusually large overflow. As far as we are aware, however, asymptotic behavior of the overflow has  not been systematically investigated.

More precisely, we consider the following model: For any $n\ge 1$,  let $X_{n,1},\ldots,X_{n,n}$ be  iid rv's with values in $ M_n\subset \N:=\{1,2,\ldots\}$ and let $p_{n,m}=\P(X_{n,1}=m), m\in M_n$, be the common distribution among the boxes for each of the $n$ balls in the $n$th experiment. Let also
\begin{equation*}
N_{n,k}(m)=\sum_{j=1}^{k-1}\,I_{\{X_{n,j}=m\}},
\end{equation*}
for any $n\in \N$, $k\in\{1,\ldots,n,n+1\}$ and $m\in M_n$, where $I_{\{\cdot\}},$ denotes the indicator of the events within brackets. That is $N_{n,k}(m)$ is the number of balls among first $k-1$ balls for which the
$m$th box was selected.

Let $r$ be a given positive integer, which denotes the (same) capacity of every container. Then
\begin{equation}
\label{eq:Ynk}
Y_{n,k}=\sum_{m\in M_n}\,I_{\{X_{n,k}=m\}}\,I_{\{N_{n,k}(m)\ge r \}}
\end{equation}
is 1 if the $k$th ball lands in the overflow, and is 0 otherwise. Naturally, $Y_{n,k}=0$ for $k=1,\ldots,r$. Consequently, the size of the overflow, denoted $V_{n,r}$, can be written as
\begin{equation*}
V_{n,r}=\sum_{k=1}^n\,Y_{n,k}.
\end{equation*}
We are interested in the asymptotic distribution of $V_{n,r}$, as $n\to\infty$. We will show that there are regimes relating $(p_{n,m})_{m\in M_n}$ and $n\to\infty$ under which the limiting distribution of $V_{n,r}$ (possibly standardized) is either Poisson or normal.
These regimes will be defined through the limiting behavior of $$p_n^*:=\sup_{m\in M_n}\,p_{n,m}\qquad\mbox{and}\qquad\sum_{m\in M_n}\,p_{n,m}^{r+1}.$$
Actually,  we impose assumptions on $\lim\limits_{n\to\infty}\,np_n^*$ and  $\lim\limits_{n\to\infty}\,n^{r+1}\,\sum_{m\in M_n}\,p_{n,m}^{r+1}$.
\subsection{Multinomial distribution and negative association}
\label{neg_dep_rem}
Note that, for distinct $m_1,\ldots,m_s\in M_n$ and any $k=1,\ldots,n$, $(N_{n,k}(m_1),\ldots,N_{n,k}(m_s))$ has  multinomial distribution  $\mathrm{Mn}_s(k-1; p_{n,m_1},\ldots,p_{n,m_s})$. In particular, $N_{n,k}(m)$  has the binomial distribution
$\Bin(k-1, p_{n,m})$, that is,
\begin{equation}
\label{eq:binom}
\P(N_{n,k}(m)=i)={k-1\choose i}p_{n,m}^iq_{n,m}^{k-1-i},\;i=0,\ldots,k-1,
\end{equation}
where $q_{n,m}=1-p_{n,m}$. Also, let
\begin{equation*}
N_{n,k}^l(m)=N_{n,l}(m)-N_{n,k}(m)=\sum_{j=k}^{l-1}\,I_{\{X_{n,j}=m\}},
\end{equation*}
for $k<l$, and $N_{n,k}^l(m)=0$, for $k\ge l$.
Then, for distinct $j_1,\ldots,j_t\in M_n$ and $k<l$,  $(N_{n,k}^l(j_1),\ldots,N_{n,k
}^l(j_t))$ has multinomial distribution  $\mathrm{Mn}_t(l-k; p_{n,j_1},\ldots,p_{n,j_t})$. Moreover,   vectors $(N_{n,k}(m_1),\ldots,N_{n,k}(m_s))$ and $(N_{n,k}^l(j_1),\ldots,N_{n,k}^l(j_t))$ are independent.
Further, it is well known that multinomial random variables are negatively orthant dependent (NOD), that is, for $m_1\neq m_2$
\begin{equation}
\label{eq:NOD1}
\small\P(N_{n,k}(m_1)\ge x_1,\,N_{n,k}(m_2)\ge x_2)\le \P(N_{n,k}(m_1)\ge x_1)\,\P(N_{n,l}(m_2)\ge x_2).
\end{equation}

As such they are also negatively associated (NA) - see \cite{JP} for the definition and basic properties $P_1,\ldots, P_7$.

In particular, both sets $N_{n,k}(m_1),\ldots, N_{n,k}(m_t)$ and $N_{n,k}^l(j_1),\ldots,N_{n,k}^l(j_t)$ are NA and, by property $P_7$, the combined set of $N_{n,k}$ and $N_{n,k}^l$ variables is also NA. In particular, by $P_4$, for distinct $m_1,m_2,n_1, n_2$,  the subset $N_{n,k}(m_1)$, $N_{n,k}(n_1)$, $N_{n,k}(m_2)$, $N_{n,k}(n_2)$, $N_{n,k}^l(m_2)$ $N_{n,k}^l(n_2)$ is NA as well. Finally, noting that $N_{n,l}(m)=N_{n,k}(m)+N_{n,k}^l(m)$ we conclude by $P_6$ that $N_{n,k}(m_1)$, $N_{n,k}(n_1)$, $N_{n,l}(m_2)$, $N_{n,l}(n_2)$ are NA.

Consequently, the following extended versions of the NOD property \eqref{eq:NOD1} hold:
\begin{equation}
\label{eq:PA2}
\small\begin{split}
\P&(N_{n,k}(m_1)\ge x_1, N_{n,k}(n_1)\ge y_1,N_{n,l}(m_2)\ge x_2,N_{n,l}(n_2)\ge y_2)\\
&\le
\P(N_{n,k}(m_1)\ge x_1)\P(N_{n,k}(n_1)\ge y_1)\P(N_{n,l}(m_2)\ge x_2)\P(N_{n,l}(n_2)\ge y_2)
\end{split}
\end{equation}
and, taking $y_1=y_2=0$  in \eqref{eq:PA2},
\begin{equation}
\label{eq:covH}
\small\P(N_{n,k}(m_1)\ge x_1,\,N_{n,l}(m_2)\ge x_2)\le \P(N_{n,k}(m_1)\ge x_1)\,\P(N_{n,l}(m_2)\ge x_2).
\end{equation}
\subsection{Auxiliary random variables}
We find it convenient to introduce sequences of random variables $(X_n)$ and $(Y_n)$ such that, for any $n\in\N$, the random variables $X_n, Y_n,X_{n,1},\ldots,X_{n,n}$ are iid. This allows, in general, to simplify expressions
because sums over $m\in M_n$ can be represented as expectations
and computations are compactly carried out by means of conditional expectations. For example,
\begin{equation*}
\sum_{m\in M_n}\,p_{n,m}^{r+1}=\E\,p^{r}_{X_n},
\end{equation*}
where here and everywhere below we write $\pxn$ for $p_{n,X_n}$.

Let $\cF_{n,k} = \sigma(X_{n,1},\ldots,X_{n,k})$ be the $\sigma$-algebra generated by $X_{n,1},\ldots,X_{n,k}$, for $k=1,\ldots,n$, and note that
$N_{n,j}(m)$ is $\cF_{n,k}$-measurable, for any $m\in M_n, k\ge j-1$. Note also that, for any $n,k$, $X_n$ is independent of $\cF_{n,k}$. Then $Y_{n,j}$  can be written as
\begin{equation}
\label{eq:form0}
Y_{n,j}=\E\Big(\tfrac{I_{\{X_{n,j}=X_n\}}}{\pxn}I_{\{N_{n,j}(X_n)\ge r\}}\Big|\cF_{n,n}\Big).
\end{equation}
So, for $j\ge k$,
\begin{equation}
\label{eq:form}
\begin{split}
\E\,(Y_{n,j}|\cF_{n,k-1})&=\E\Big(\tfrac{I_{\{X_{n,j}=X_n\}}}{\pxn}I_{\{N_{n,j}(X_n)\ge r\}}\Big|\cF_{n,k-1}\Big)\\
&=\E\Big(\E\Big(\tfrac{I_{\{X_{n,j}=X_n\}}}{\pxn}I_{\{N_{n,j}(X_n)\ge r\}}\Big|X_n,\cF_{n,j-1}\Big)\Big|\cF_{n,k-1}\Big)\\
&=\E\Big(\E\Big(\tfrac{I_{\{X_{n,j}=X_n\}}}{\pxn}\Big|X_n,\cF_{n,j-1}\Big)\,I_{\{N_{n,j}(X_n)\ge r\}}\Big|\cF_{n,k-1}\Big)\\
&=\E(I_{\{N_{n,j}(X_n)\ge r\}}|\cF_{n,k-1}).
\end{split}
\end{equation}
Hence, $\E\,(Y_{n,j}|\cF_{n,k})=\E(I_{\{N_{n,j}(X_n)\ge r\}}|\cF_{n,k})$, for $j>k$, and  $\E\,(Y_{n,k}|\cF_{n,k})=Y_{n,k}$.

Note that representation \eqref{eq:form} implies
\begin{equation}
\label{eq:form1}
\E(Y_{n,k}|\cF_{n,k-1})=\P(N_{n,k}(X_n)\ge r|\cF_{n,k-1})=\P(N_{n,k}(X_n)\ge r|\cF_{n,n}).
\end{equation}
Taking expectations of both extremes of \eqref{eq:form} we get
\begin{equation}
\label{eq:expe}
\E\,Y_{n,j}=\E\,\P(N_{n,j}(X_n)\ge r|\cF_{n,k-1})=\E\sum_{i=r}^{j-1}\,\binom{j-1}{i}\,p_{X_n}^iq_{X_n}^{j-1-i},
\end{equation}
where $\qxn=1-\pxn$. Furthermore, for $k,l=1\ldots,n$, \eqref{eq:form1} yields
{\small\begin{equation*}
	\E(\,\E\,(Y_{n,k}|\cF_{n,k-1})\,\E\,(Y_{n,l}|\cF_{n,l-1}))=\E\,(\P(N_{n,k}(X_n)\ge r|\cF_{n,n})\P(N_{n,l}(Y_n)\ge r|\cF_{n,n}))
	\end{equation*}
	and, because $N_{n,k}(X_n)$ and $N_{n,l}(Y_n)$ are conditionally independent given $\cF_{n,n}$, it follows that
	\begin{equation*}
	\E(\,\E\,(Y_{n,k}|\cF_{n,k-1})\,\E\,(Y_{n,l}|\cF_{n,l-1}))
	=\P(N_{n,k}(X_n)\ge r,N_{n,l}(Y_n)\ge r).
	\end{equation*}}
Consequently, for any $k,l$,
\begin{equation}
\label{eq:war}
\small\cov\left(\E\,(Y_{n,k}|\cF_{n,k-1}),\E\,(Y_{n,l}|\cF_{n,l-1})\right)=\cov(I_{\{N_{n,k}(X_n)\ge r\}},I_{\{N_{n,l}(Y_n)\ge r\}}).
\end{equation}

\section{Poissonian asymptotics}
Let $\mathrm{Pois}(\mu)$ denote the Poisson distribution with parameter $\mu$.
\begin{theorem}\label{thm:rrr}
	Let $\mu>0$. If
	\bel{eq:inten1}
	n^{r+1}\E\,p_{X_n}^r\to\,(r+1)!\mu
	\ee
	and
	\bel{eq:jedno}
	n\,p_n^*\to\,0,
	\ee
	then $V_{n,r}\stackrel{d}{\to}\,\mathrm{Pois}(\mu)$.
\end{theorem}
{\noindent\bf Examples:}\\
$\bullet\ $ Consider the uniform case, that is, $p_{n,j}=1/m_n$, for $j\in M_n=\{1,\ldots,m_n\}$. Then by the above theorem we get
$$\tfrac{n^{r+1}}{m_n^r}\to (r+1)!\mu \qquad \Rightarrow  \qquad V_{n,r}\stackrel{d}{\to}\,\mathrm{Pois}(\mu).$$
Illustrative simulations are visualized in Figure \ref{fig:one}.

\begin{figure}
	\centering
	\includegraphics[height=60mm,width=100mm]{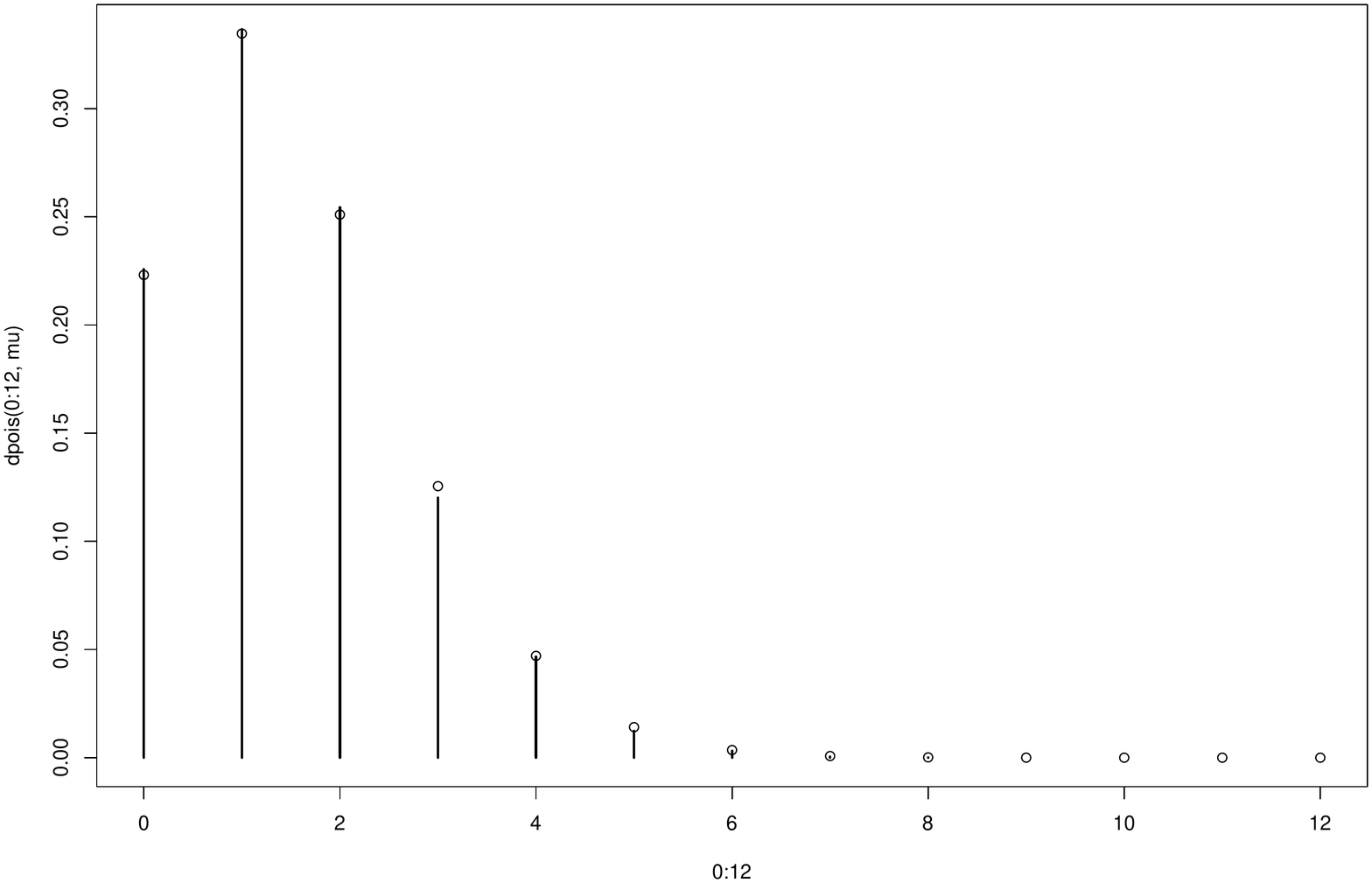}
	\caption{\label{fig:one}Simulations of the overflow in the uniform case with $r=2$, $n=10^6$,  $m_n=a n^{(r+1)/r}$ with $a=1/3$ (i.e. $m_{10^6}=10540926$, and $\mu=1.5$) are shown as vertical lines ($10^4$ repetitions) while Poisson probabilities  for $k=0,\ldots,12$, $\mathrm{dpois}(0:12,\mu)$, are depicted by circles.}
\end{figure}
$\bullet\ $Consider the geometric case, $p_{n,j}=p_n(1-p_n)^j$, $j\ge 0$.
Then
\bel{geo}
n^{r+1}\E\,p_{X_n}^r=\tfrac{(np_n)^{r+1}}{1-(1-p_n)^{r+1}}.
\ee
Take $p_n=n^{-\tfrac{r+1}{r}}$ (that is $n^{r+1}p_n^r=1$). Thus, by \eqref{geo},
$
n^{r+1}\E\,p_{X_n}^r=\tfrac{p_n}{(r+1)p_n+o(p_n)}\to \tfrac{1}{r+1}.
$
Moreover, $np_n^*=np_n=n^{-\tfrac{1}{r}}\to 0$.

Consequently, the above theorem  yields $V_{n,r}\stackrel{d}{\to}\,\mathrm{Pois}(\mu)$ with $\mu = \tfrac{1}{(r+1)!(r+1)}$. Illustrative simulations are visualized in Figure \ref{fig:two}.
\begin{figure}
	\centering
	\includegraphics[height=60mm,width=100mm]{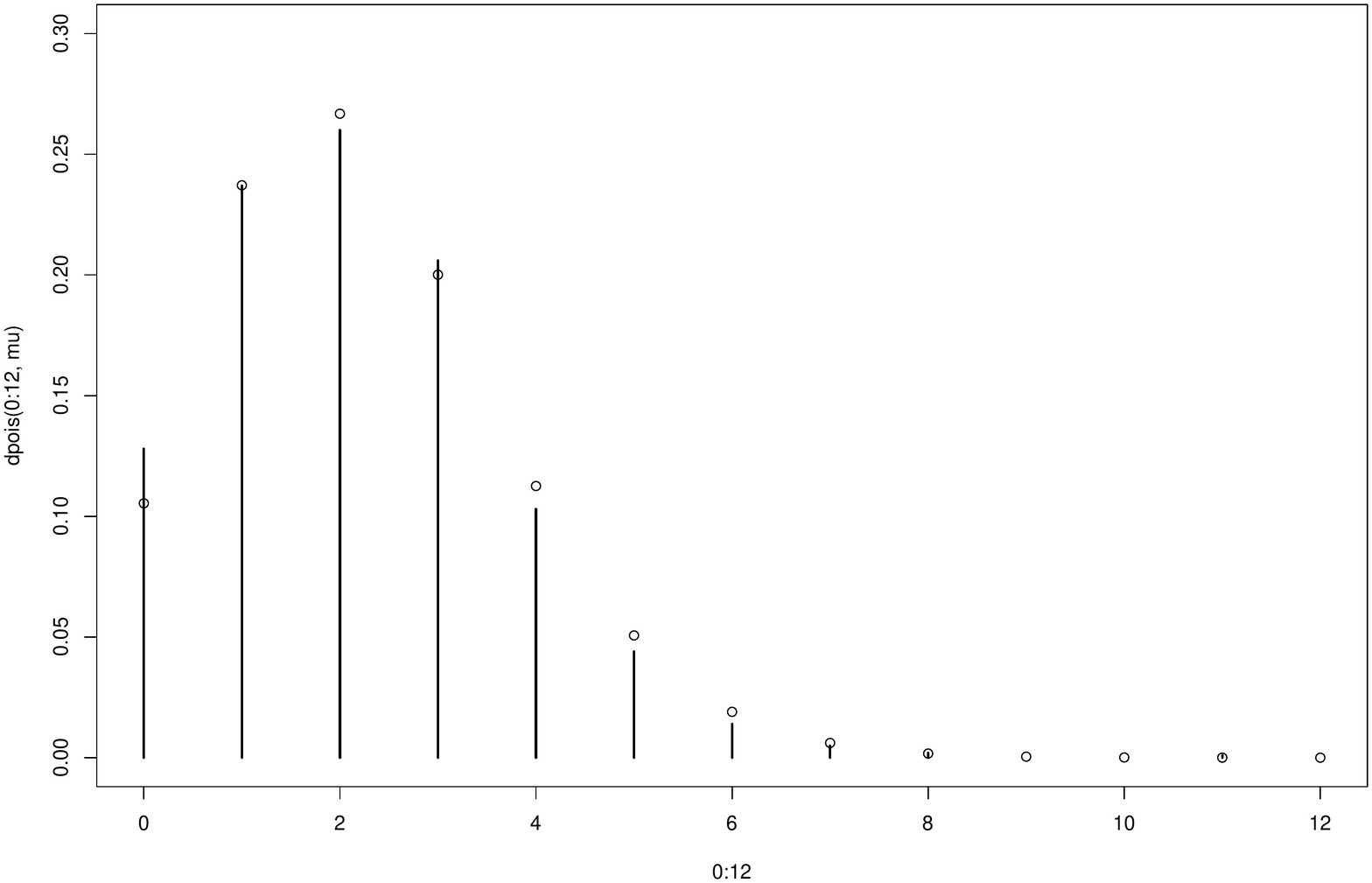}
	\caption{\label{fig:two}Simulations of the overflow in the geometric case with $r=3$, $n=10^6$,  $p_n=a n^{-(r+1)/r}$ with $a=6$ (i.e. $p_{10^6}\approx 1.3\times10^{-6}$ and  $\mu=2.25$) are shown as vertical lines ($10^3$ repetitions) while Poisson probabilities for $k=0,\ldots,12$, $\mathrm{dpois}(0:12,\mu)$, are depicted by circles.}
\end{figure}

The method of Poissonization and de--Poissonization was used in \cite[Theorem~8.2]{HJ} to prove Theorem \ref{thm:rrr}, for $r=1$. The proof we present here is  entirely different and relies on the following martingale-type convergence result from \cite{BKS}.

\begin{theorem}\label{thm:BKS}
	Let $\{Y_{n,k},\,k=1,\ldots,n;\,n\ge 1\}$ be a double sequence of non-negative random variables, adapted to a row-wise increasing double sequence of $\sigma$-fields $\{\cF_{n,k},\,k=1,\ldots,n;\,n\ge 1\}$, and let $\eta>0$. If
	\bel{eq:BKS1}
	\max_{1\le k\le n}\,\E\,(Y_{n,k}|\cF_{n,k-1})\,\stackrel{\P}{\to}0,
	\ee
	\bel{eq:BKS2}
	\sum_{k=1}^n\,\E\,(Y_{n,k}|\cF_{n,k-1})\,\stackrel{\P}{\to} \eta
	\ee
	and, for any $\eps>0$,
	\bel{eq:BKS3}
	\sum_{k=1}^n\,\E\,(Y_{n,k}I_{\{|Y_{n,k}-1|>\eps\}}|\cF_{n,k-1})\,\stackrel{\P}{\to} 0,
	\ee
	then
	$\sum_{k=1}^n\,Y_{n,k}\stackrel{d}{\to}\mathrm{Pois}(\eta)$.
\end{theorem}
In the proof of Theorem~\ref{thm:rrr} we use the following consequences of \eqref{eq:inten1} and \eqref{eq:jedno}.
\begin{lemma}
	\label{lem:to_zero} Let $s$ be a positive integer.  If \eqref{eq:inten1} and \eqref{eq:jedno} hold, then
	\begin{equation}
	\label{eq:zero1}
	n^s\E\,p_{X_n}^s\to 0
	\end{equation}
	and
	\begin{equation}
	\label{eq:zero2}
	n^{s+1}\E\,p_{X_n}^s\to0,\quad s>r.
	\end{equation}
\end{lemma}

\begin{proof}
	Since $n^s\E\,p_{X_n}^s\le (np_n^*)^{s}$, \eqref{eq:zero1} follows from \eqref{eq:jedno}. Also, \eqref{eq:zero2} follows from \eqref{eq:inten1} and \eqref{eq:jedno} since $$n^{s+1}\E\,p_{X_n}^s\le (np_n^*)^{s-r}n^{r+1}\E\,p_{X_n}^r.$$
\end{proof}
We also need the simple estimate shown below, for the tail of a binomial sum.
\begin{lemma}
	\label{lem:bin_bdd}
	Let $m,n$ be positive integers, such that $m\le n$, and let $p\in(0,1)$. Then
	\begin{equation}
	\label{eq:binineq}
	\sum_{i=m}^n{n\choose i}p^i(1-p)^{n-i}\le \frac{(np)^m}{m!}.
	\end{equation}
\end{lemma}
\begin{proof}
	The left-hand side of \eqref{eq:binineq} is  $\P(B_n\ge m)$, where $B_n$ has distribution $\Bin(n,p)$. Arguing by induction on $n$, we have
	\begin{equation*}
	\begin{split}
	\P(B_{n+1}\ge m)&=\P(B_n\ge m-1)p+\P(B_n\ge m)(1-p)\\
	&\le \frac{(np)^{m-1}}{(m-1)!}p+\frac{(np)^m}{m!}(1-p)\\
	&\le \frac{((n+1)p)^m}{m!},
	\end{split}
	\end{equation*}
	where the last inequality follows from $mn^{m-1}+n^m(1-p)\le(n+1)^m$.
\end{proof}

\section{Proof of Theorem \ref{thm:rrr}}
\begin{proof}
	We  show that for $Y_{n,k}$ defined in \eqref{eq:Ynk}, conditions \eqref{eq:BKS1}, \eqref{eq:BKS2} with $\eta=\mu$, and \eqref{eq:BKS3} are satisfied. First we note that \eqref{eq:BKS3} is trivially satisfied because, for $\eps<1$, $Y_{n,k}=0$ if and only if $I_{\{|Y_{n,k}-1|>\eps\}}=1$.

	The rest of the proof is divided into three steps. In {\bf Step  I} we check that \eqref{eq:BKS1} is satisfied. Then we prove that \eqref{eq:BKS2} holds in quadratic mean, that is, $$\E\Big(\sum_{k=1}^n\,\E\,(Y_{n,k}|\cF_{n,k-1})- \mu\Big)^2\to0.$$ To that end we show  that $\sum_{k=1}^n\,\E\,Y_{n,k}\to\mu$ and $\var\sum_{k=1}^n\E\,(Y_{n,k}|\cF_{n,k-1})\to 0$ in {\bf Step II} and {\bf Step III}, respectively.
	
	{\bf\noindent Step  I}: We prove \eqref{eq:BKS1} using \eqref{eq:form1}. Clearly, $I_{\{N_{n,k}(m)\ge r\}}\le I_{\{N_{n,l}(m)\ge r\}}$, for $k\le l$, so
	\begin{equation*}
	\max_{1\le k\le n}\,\E\,(Y_{n,k}|\cF_{n,k-1})=\E\,(Y_{n,n}|\cF_{n,n-1}).
	\end{equation*}
	Note also that, due to \eqref{eq:expe}, \eqref{eq:binineq} and \eqref{eq:zero1},
	\begin{equation*}
	\E Y_{n,n}=\E\,\sum_{i=r}^{n-1}{n-1\choose i}p_{X_n}^{i}q_{X_n}^{n-1-i}\le n^r\E\,p_{X_n}^r\to0.
	\end{equation*}
	Consequently, Markov's inequality implies $\E\,(Y_{n,n}|\cF_{n,n-1})\stackrel{\P}{\to}0$ and thus \eqref{eq:BKS1} follows.
	
	{\bf Step II}:
	To prove that $\lim_n\sum_{k=1}^n\E\,Y_{n,k}= \mu$ we show that $\limsup_n$ and $\liminf_n$ are respectively bounded above and below by $\mu$.
	From \eqref{eq:expe}, \eqref{eq:binineq} and \eqref{eq:inten1}
	\begin{equation*}
	\begin{split}
	\sum_{k=1}^n\E\, Y_{n,k}&=\E\,\sum_{k=1}^n\sum_{i=r}^{k-1}{k-1\choose i}p_{X_n}^{i}q_{X_n}^{k-1-i}\\
	&\le\E\,\sum_{k=1}^n\frac{(k-1)^rp_{X_n}^r}{r!}\le \frac{n^{r+1}}{(r+1)!}\E\,p_{X_n}^r\to\mu,
	\end{split}
	\end{equation*}
	so $\limsup_n\,\sum_{k=1}^n\,\E\,Y_{n,k}\le \mu$.
	
	Additionally, since by \eqref{eq:expe},  $\E\,Y_{n,k}=\P(N_{n,k}(X_n)\ge r)\ge \P(N_{n,k}(X_n)=r)$ and $(1-p)^k\ge 1-kp, p\in(0,1)$,  we have
	\begin{equation}
	\label{eq:ineq}
	\begin{split}
	\sum_{k=1}^n\,\E\,Y_{n,k}&\ge \sum_{k=1}^n{k-1\choose r}\E\,p_{X_n}^{r}q_{X_n}^{k-1-r}\\
	&\ge \E\,p_{X_n}^{r}\,\sum_{k=1}^n{k-1\choose r}-\E\,p_{X_n}^{r+1}\,\sum_{k=1}^n{k-1\choose r}(k-1-r).
	\end{split}
	\end{equation}
	Further, observe  that
	\begin{equation*}
	\frac{(r+1)!}{n^{r+1}}\sum_{k=1}^n{k-1\choose r}\to1\quad\text{and}\quad \frac{r!(r+2)}{n^{r+2}}\sum_{k=1}^n{k-1\choose r}(k-1-r)\to1.
	\end{equation*}
	Thus, by \eqref{eq:inten1}  and \eqref{eq:zero2}, the rhs of  \eqref{eq:ineq} converges to $\mu$ and  so, $\liminf\limits_n\sum\limits_{k=1}^n\E\,Y_{n,k}\ge \mu$.
	
	{\bf Step III}:
	We prove that $W_n:=\var\,\sum_{k=1}^n\,\E\,(Y_{n,k}|\cF_{n,k-1})\to 0$, relying on the NOD property of $N_{n,k}(m_1), N_{n,l}(m_2)$, for distinct $m_1,m_2\in M_n$. In what follows we compute and bound some expectations that add up to $W_n$. First note from \eqref{eq:war} that 
	\begin{equation*}
	\begin{split}
	W_n&=\sum_{k,l=1}^{n}\cov(\E\,(Y_{n,k}|\cF_{n,k-1}),\E\,(Y_{n,l}|\cF_{n,l-1}))\\
	&=\sum_{k,l=1}^{n}\cov(I_{\{N_{n,k}(X_n)\ge r\}},I_{\{N_{n,l}(Y_n)\ge r\}}).
	\end{split}
	\end{equation*}
	For $U,V$ square-integrable random variables and $\cal G$ a $\sigma$-algebra, let the conditional covariance be defined as
	\begin{equation*}
	\cov(U,V|{\cal G})=\E(UV|{\cal G})-\E(U|{\cal G})\E(V|{\cal G}).
	\end{equation*}
	Also, let $I_k(m)=I_{\{N_{n,k}(m)\ge r\}}$ (for simplicity) and $k\wedge l=\min\{k,l\}$. Then, by the iid assumption of $X_{n,1},\ldots,X_{n,n}, X_n,Y_n$, we have
	\begin{equation}
	\label{eq:condcov}
	\cov(I_k(X_n),I_l(Y_n)|X_n,Y_n)=\E(I_k(X_n)I_l(Y_n)|X_n,Y_n)-\E(I_k(X_n)|X_n)\E(I_l(Y_n)|Y_n).
	\end{equation}
	Furthermore,
	\begin{equation}
	\label{eq:X=Y}
	\begin{split}
	\E(I_k(X_n)I_l(Y_n)|X_n,Y_n)I_{\{X_n=Y_n\}}&=\E(I_{\{X_n=Y_n\}}I_k(X_n)I_l(Y_n)|X_n,Y_n)\\
	&=\E(I_{\{X_n=Y_n\}}I_k(X_n)I_l(X_n)|X_n,Y_n)\\
	&=\E(I_{k\wedge l}(X_n)|X_n)I_{\{X_n=Y_n\}},
	\end{split}
	\end{equation}
	where the last equality follows from $I_k(m)\le I_l(m)$, for $k\le l$, because $N_{n,k}(m)\ge r$ implies $N_{n,l}(m)\ge r$. So, from \eqref{eq:condcov} and \eqref{eq:X=Y}, we get
	\begin{equation}
	\label{eq:condcoveq}
	\cov(I_k(X_n),I_l(Y_n)|X_n,Y_n)I_{\{X_n=Y_n\}}\le\E(I_{k\wedge l}(X_n)|X_n)I_{\{X_n=Y_n\}}.
	\end{equation}
	Furthermore, by the NOD property \eqref{eq:covH},
	\begin{equation}
	\label{eq:Xnot=Y}
	\begin{split}
	\E(I_k(X_n)I_l(Y_n)|X_n,Y_n)&I_{\{X_n\ne Y_n\}}=\E(I_{\{X_n\not=Y_n\}}I_k(X_n)I_l(Y_n)|X_n,Y_n)\\
	&\le\E(I_{\{X_n\not=Y_n\}}I_k(X_n)|X_n,Y_n)\E(I_{\{X_n\not=Y_n\}}I_l(Y_n)|X_n,Y_n)\\
	&=\E(I_{k}(X_n)|X_n)\E(I_{l}(Y_n)|Y_n)I_{\{X_n\not=Y_n\}}.
	\end{split}
	\end{equation}
	Hence, from \eqref{eq:condcov} and \eqref{eq:Xnot=Y}, we have
	\begin{equation}
	\label{eq:condcovnoteq}
	\cov(I_k(X_n),I_l(Y_n)|X_n,Y_n)I_{\{X_n\not=Y_n\}}\le0.
	\end{equation}
	And, finally, from \eqref{eq:condcoveq} and \eqref{eq:condcovnoteq},
	\begin{equation}
	\label{eq:condcovbound}
	\cov(I_k(X_n),I_l(Y_n)|X_n,Y_n)\le\E(I_{k\wedge l}(X_n)|X_n)I_{\{X_n=Y_n\}},
	\end{equation}
	which, after taking expectation, yields
	\begin{equation}
	\label{eq:covbound}
	\cov(I_k(X_n),I_l(Y_n))\le\E(I_{k\wedge l}(X_n)I_{\{X_n=Y_n\}})=\E(I_{k\wedge l}(X_n)\pxn).
	\end{equation}
	Also, by \eqref{eq:binineq},
	\begin{equation*}
	\E(I_{k\wedge l}(X_n)\pxn|X_n)= \sum_{i=r}^{k\wedge l-1}{k\wedge l-1\choose i}p_{X_n}^{i}q_{X_n}^{k\wedge l-1-i}\pxn\le(k-1)^rp^{r+1}_{X_n}.
	\end{equation*}
	Last, taking expectation above and adding over $k$ and $l$, from \eqref{eq:covbound} we obtain
	\begin{equation*}
	W_n\le \sum_{k,l=1}^n(k-1)^r\E\, p_{X_n}^{r+1}\le n^{r+2}\E\,p_{X_n}^{r+1}\to0,
	\end{equation*}
	where convergence to 0 follows from \eqref{eq:zero2}. Finally, since $W_n\ge 0$, it follows that $W_n\to 0$.	
\end{proof}
\section{Normal asymptotics for overflow}

The following theorem gives conditions under which the overflow is asymptotically normal.

\begin{theorem}
	\label{norasy}
	Assume that $np_n^*\to\lambda \ge 0$ and that $n^{r+1}\E\,p_{X_n}^r\to\infty$. Then
	$$
	\frac{V_{n,r}-\E\,V_{n,r}}{\sqrt{\var\,V_{n,r}}}\stackrel{d}{\to} \mathrm{N}(0,1).
	$$
\end{theorem}

{\bf \noindent Examples}\\
\noindent$\bullet$ Consider the uniform case, i.e. $p_{n,j}=1/m_n$, $j\in M_n=\{1,\ldots,m_n\}$. Then by the above theorem we get
\begin{equation*}
\frac{n^{r+1}}{m_n^r}\to \infty \quad \mbox{and}\quad \frac{n}{m_n}\to \lambda\ge 0 \qquad\Longrightarrow \qquad \frac{V_{n,r}-\E\,V_{n,r}}{\sqrt{\var\,V_{n,r}}}\stackrel{d}{\to} \mathrm{N}(0,1).
\end{equation*}
Note that $m_n=\kappa n^a$ with $a\in[1,\,1+r^{-1})$ yields normal asymptotics.\\
\noindent$\bullet$ Consider the geometric case, $p_{n,j}=p_n(1-p_n)^j$, $j\ge 0$, with $p_n=n^{-a}$ and $a\in[1,\,1+r^{-1})$. Then \eqref{geo} yields
$$
n^{r+1}\E\,p_{X_n}^r=\tfrac{n^{r+1-ra}}{r+1+o(1)}\to \infty.
$$
Moreover,
$$
np_n^*=np_n=n^{1-a}\to \left\{\begin{array}{ll} 1, & a=1, \\
0, & 1<a<1+r^{-1}.\end{array} \right.
$$
Thus, asymptotic normality of $V_{n,r}$ follows from the above theorem. Illustrative simulations are visualized in Figures \ref{fig:three} and \ref{fig:four}.

\begin{figure}
	\centering
	\includegraphics[height=60mm,width=100mm]{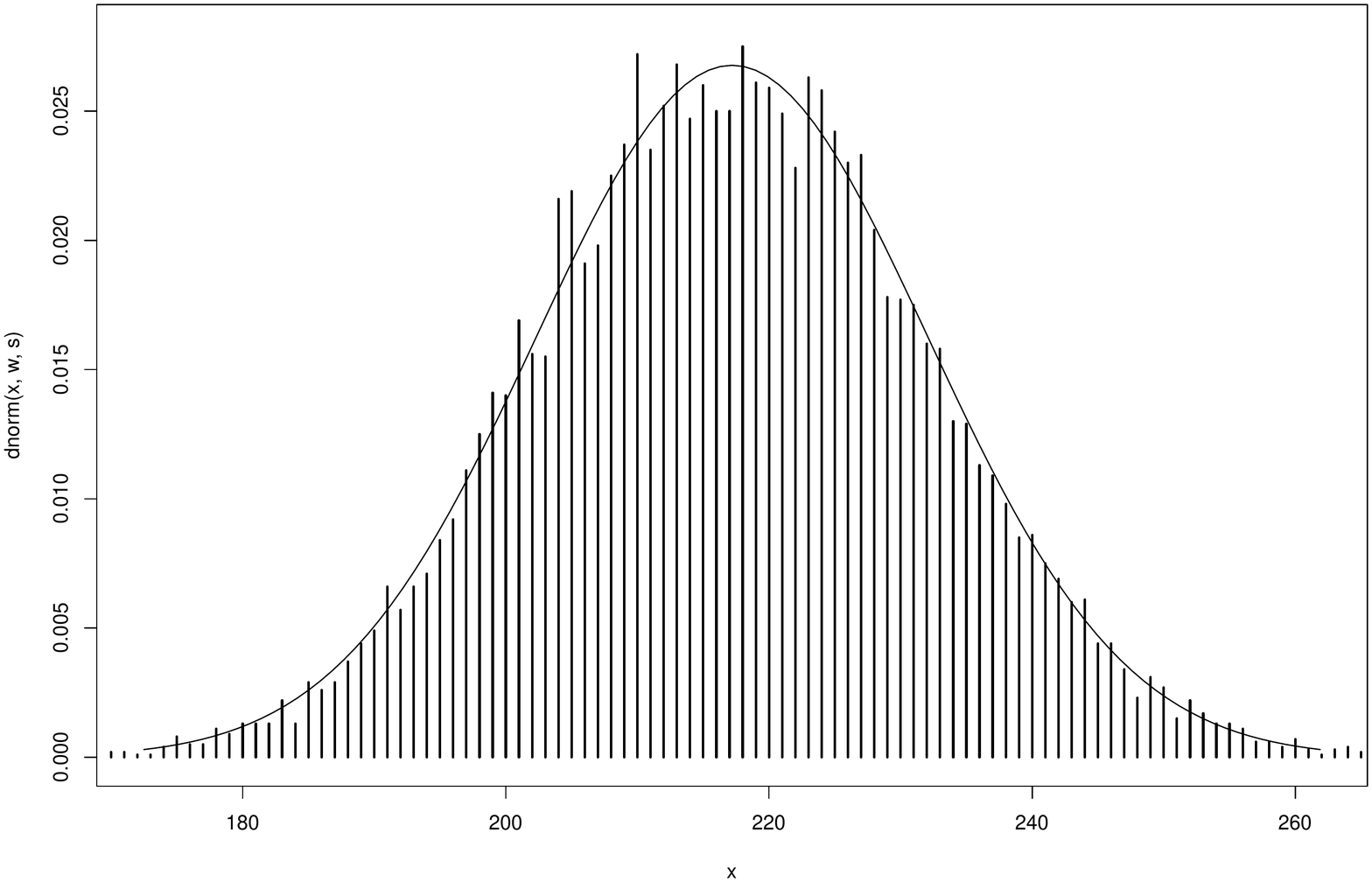}
	\caption{\label{fig:three} Simulations of the overflow in the uniform case with $r=2$, $n=10^4$,  $m=n^a$ with $a=1.1$ (i.e. $m=25118$)  are shown as vertical lines ($10^4$ repetitions) vs. the graph of the normal density $\mathrm{dnorm}(x,w,s)$, where $w=217.2$ and $s=14.9$ are empirical mean and standard deviation, respectively.}
\end{figure}

\begin{figure}
	\centering
	\includegraphics[height=60mm,width=100mm]{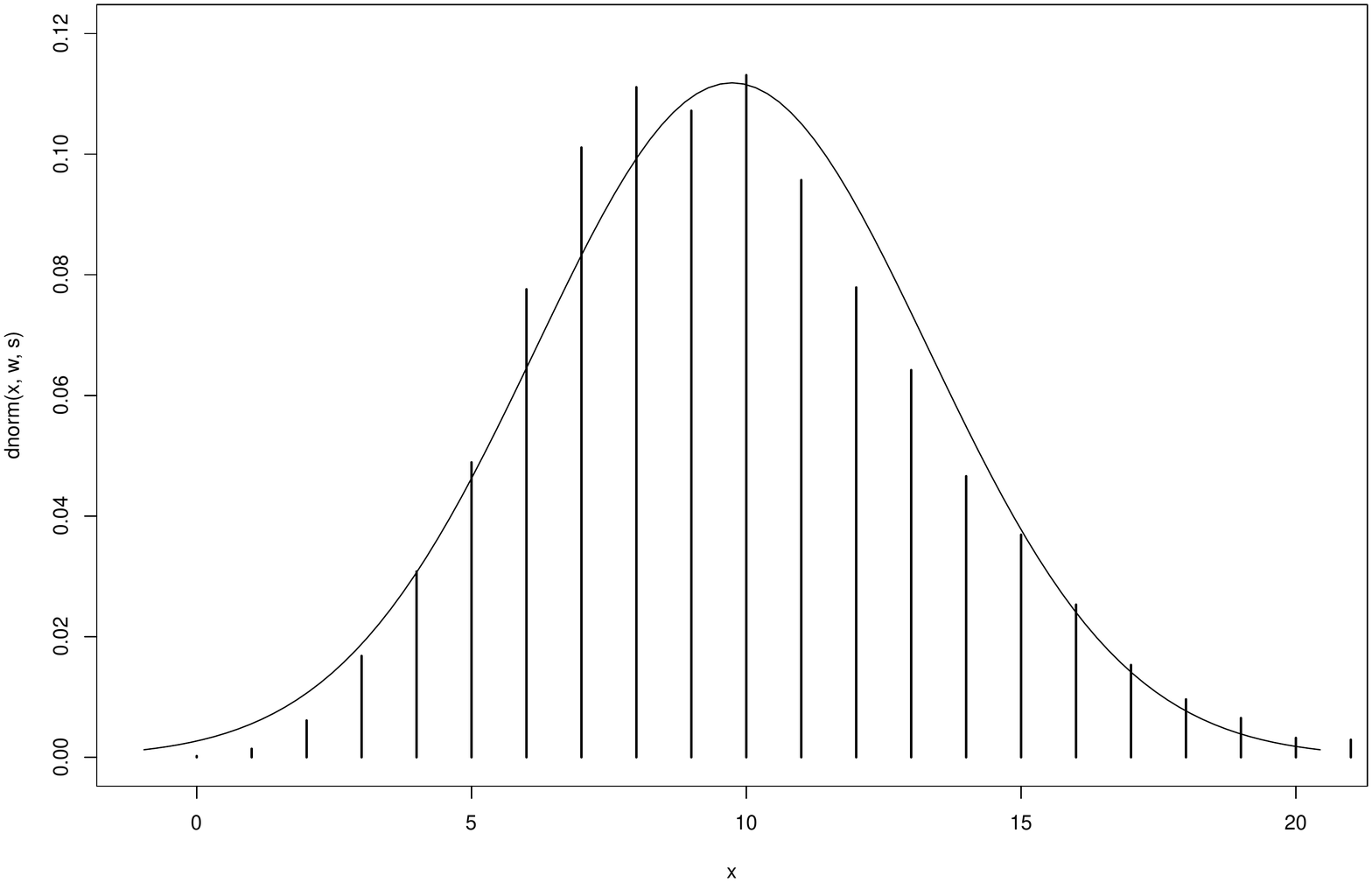}
	\caption{\label{fig:four} Simulations of the overflow in the geometric case with $r=4$, $n=10^4$,  $a=1$  are shown as vertical lines ($10^4$ repetitions) vs. the graph of the normal density $\mathrm{dnorm}(x,w,s)$, where $w=9.74$  and $s=3.57$ are empirical mean and standard deviation, respectively.}
\end{figure}

The proof of Theorem \ref{norasy} is split in several steps given in four subsections below. In Subsection 4.1 we decompose $V_{n,r}-\E\,V_{n,r}$ in the sum of martingale differences $\sum_{k=1}^n\,d_{n,k}$, with suitably defined (uniformly bounded) $d_{n,k}$'s. In Subsection 4.2 we show  that $\var\,V_{n,r}$ is of order $n^{r+1}\E\,p_{X_n}^r$. In Subsection 4.3 we show that $\var\,\sum_{k=1}^n\var\,(d_{n,k}|\cF_{n,k-1})$ is of order $o((n^{r+1}\E\,p_{X_n}^r)^2)$. The final part of the proof, which gathers all previous steps, is given in Subsection \ref{sec:finaltouch}.

\subsection{Martingale differences decomposition}
\label{sec:martindec}
\begin{lemma}
	The centered size of the overflow can be represented as $V_{n,r}-\E\,V_{n,r}=\sum_{k=1}^nd_{n,k}$, where the $d_{n,k}$ are martingale differences defined by
	\begin{equation}
	\label{eq:dnk}
	d_{n,k}=\sum_{j=k}^n\left(\E\,(Y_{n,j}|\cF_{n,k})-\E\,(Y_{n,j}|\cF_{n,k-1})\right).
	\end{equation}
\end{lemma}
\begin{proof}
	Clearly, $\E(d_{n,k}|\cF_{n,k-1})=0$. Further, noting that $\cF_{n,0}$ is the trivial $\sigma$-algebra,
	\begin{equation*}
	\begin{split}
	\sum_{k=1}^{n}d_{n,k}&=\sum_{j=1}^{n}\sum_{k=j}^{n}(\E\,(Y_{n,j}|\cF_{n,k})-\E\,(Y_{n,j}|\cF_{n,k-1}))\\
	&=\sum_{j=1}^{n}(\E\,(Y_{n,j}|\cF_{n,n})-\E\,Y_{n,j})= V_{n,r}-\E\, V_{n,r}.
	\end{split}
	\end{equation*}
\end{proof}
\begin{lemma} The martingales differences $d_{n,k}$ of \eqref{eq:dnk} are uniformly bounded and can be represented as
	\begin{equation}
	d_{n,k}=\E\,\Big(\tfrac{I_{\{X_{n,k}=X_n\}}-\pxn}{\pxn}\,I_{\{N_{n,k}(X_n)+N_{n,k+1}^{n+1}(X_n)\ge r\}}|\cF_{n,k}\Big).
	\end{equation}
\end{lemma}
\begin{proof}
	Let $n,r\in\N, j> k$ and note that $N_{n,j}(X_n)=N_{n,k}(X_n)+I_{\{X_{n,k}=X_n\}}+N_{n,k+1}^j(X_n)$. For simplicity let $U_j=N_{n,k}(X_n)+N_{n,k+1}^j(X_n), V=N_{n,k}(X_n)$ and $I=I_{\{X_{n,k}=X_n\}}$. Then
	\begin{equation*}
	\begin{split}
	\{U_j+I\ge r\}&=\{U_j= r-1,I=1 \}\cup\{U_j\ge r,I=1 \}\cup \{U_j\ge r,I=0 \}\\
	&=\{U_j= r-1,I=1 \}\cup\{U_j\ge r\}.
	\end{split}
	\end{equation*}
	Hence, noting that $\{V\ge r \}=\{N_{n,k}(X_n)\ge r \}\subseteq \{N_{n,j}(X_n)\ge r \}=\{U_j+I\ge r\}$, we have
	\begin{equation*}
	\{N_{n,j}(X_n)\ge r \}=\{V\ge r \}\cup \{U_j\ge r, V<r\}\cup\{U_j= r-1,I=1, V<r\}.
	\end{equation*}
	Consequently, from \eqref{eq:form}, we can write
	\begin{equation*}
	\begin{split}
	\E(Y_{n,j}|\cF_{n,k}) &=\E(I_{\{V\ge r\}}\Big|\cF_{n,k})+\E(I_{\{U_j\ge r,V<r\}}|\cF_{n,k})+\E(I_{\{U_j= r-1,V<r,I=1\}}|\cF_{n,k})\\
	&=\E(I_{\{V\ge r\}}|\cF_{n,k-1})+\E(I_{\{U_j\ge r,V<r\}}|\cF_{n,k-1})+\E(I_{\{U_j= r-1,V<r\}}I|\cF_{n,k})
	\end{split}
	\end{equation*}
	and, similarly,
	\begin{equation*}
	\begin{split}
	\E(Y_{n,j}|\cF_{n,k}) &=\E(I_{\{V\ge r\}}|\cF_{n,k-1})+\E(I_{\{U_j\ge r,V<r\}}|\cF_{n,k-1})+\E(I_{\{U_j= r-1,V<r\}}I|\cF_{n,k-1}).
	\end{split}
	\end{equation*}
	Also, note that
	\begin{equation*}
	\E(I_{\{U_j= r-1,V<r\}}I|\cF_{n,k-1})=\E(I_{\{U_j= r-1,V<r\}}\pxn |\cF_{n,k}).
	\end{equation*}	
	Therefore, for $j>k$,
	\begin{equation*}
	\E(Y_{n,j}|\cF_{n,k})-\E(Y_{n,j}|\cF_{n,k-1})=\E\left(I_{\{U_j= r-1,V<r\}}(I-\pxn) \Big|\cF_{n,k}\right).
	\end{equation*}
	Thus
	\begin{equation*}
	\begin{split}
	e_{n,k}:&=\sum_{j=k+1}^{n}(\E(Y_{n,j}|\cF_{n,k})-\E(Y_{n,j}|\cF_{n,k-1}))\\
	&=\E\Big((I-\pxn) \sum_{j=k+1}^{n}I_{\{U_j= r-1,V<r\}}\Big|\cF_{n,k}\Big).
	\end{split}
	\end{equation*}
	Observe that, for $j>k$, $\E\,\Big(\tfrac{I_{\{X_{n,j}=X_n\}}}{p_{X_n}}|X_n, \cF_{n,j-1}\Big)=1$. Then
	\begin{equation*}
	\begin{split}
	e_{n,k}&=\E\Big((I-\pxn)I_{\{V<r\}} \sum_{j=k+1}^{n}I_{\{U_j= r-1\}}\E\,\Big(\tfrac{I_{\{X_{n,j}=X_n\}}}{p_{X_n}}\Big|X_n, \cF_{n,j-1}\Big)\Big|\cF_{n,k}\Big)\\
	&=\E\Big(\E\,\Big(\tfrac{I-\pxn}{\pxn}I_{\{V<r\}} \sum_{j=k+1}^{n}I_{\{U_j= r-1\}}I_{\{X_{n,j}=X_n\}}\Big|X_n, \cF_{n,j-1}\Big)\Big|\cF_{n,k}\Big)\\
	&=\E\Big(\tfrac{I-\pxn}{\pxn}I_{\{V<r\}} \sum_{j=k+1}^{n}I_{\{U_j= r-1\}}I_{\{X_{n,j}=X_n\}}\Big|\cF_{n,k}\Big).
	\end{split}
	\end{equation*}
	Note that $\sum_{j=k+1}^{n}I_{\{U_j= r-1\}}I_{\{X_{n,j}=X_n\}}=I_{\{U_j=r-1, U_{j+1}=r,\text{ for some } j\in\{k+1,\ldots,n\}\}}$, is equal to $I_{\{U_{n+1}\ge r\}}$ on the event $\{N_{n,k}(X_n)<r\}$. That is, using the original notation,
	\begin{equation*}
	\sum_{j=k+1}^n\,I_{\{N_{n,k}(X_n)+N_{n,k+1}^j(X_n)=r-1\}}\,I_{\{X_{n,j}=X_n\}}=I_{\{N_{n,k}(X_n)+N_{n,k+1}^{n+1}(X_n)\ge r\}}
	\end{equation*}
	on the event $\{N_{n,k}(X_n)<r\}$ and so,
	\begin{equation*}
	e_{n,k}=\E\left(\tfrac{I-\pxn}{\pxn}I_{\{N_{n,k}(X_n)<r\}} I_{\{N_{n,k}(X_n)+N_{n,k+1}^{n+1}(X_n)\ge r\}}\Big|\cF_{n,k}\right).
	\end{equation*}
	Finally, since
	\begin{equation*}
	Y_{n,k}-\E\,(Y_{n,k}|\cF_{n,k-1})=\E\Big(\,\tfrac{I_{\{X_{n,k}=X_n\}}-\pxn}{\pxn}I_{\{N_{n,k}(X_n)\ge r\}}\,\Big|\cF_{n,k}\Big)
	\end{equation*}
	we conclude that
	\begin{equation}
	\begin{split}
	\label{eq:dnkk}
	d_{n,k}&=Y_{n,k}-\E\,(Y_{n,k}|\cF_{n,k-1})+e_{n,k}\\
	=&\E\Big(\tfrac{I_{\{X_{n,k}=X_n\}}-\pxn}{\pxn}\Big(I_{\{N_{n,k}(X_n)\ge r\}}+I_{\{N_{n,k}(X_n)<r\}} I_{\{N_{n,k}(X_n)+N_{n,k+1}^{n+1}(X_n)\ge r\}}\Big)\Big|\cF_{n,k}\Big)\\
	=&\E\Big(\,\tfrac{I_{\{X_{n,k}=X_n\}}-\pxn}{\pxn}I_{\{N_{n,k}(X_n)+N_{n,k+1}^{n+1}(X_n)\ge r\}}\,\Big|\cF_{n,k}\Big).
	\end{split}
	\end{equation}
	For the boundedness of $d_{n,k}$ note that
	\begin{equation*}
	|d_{n,k}|\le \E\Big(\,\Big|\tfrac{I_{\{X_{n,k}=X_n\}}-\pxn}{\pxn}\Big||\cF_{n,k}\Big)\le \sum_{m\in M_n}|I_{\{X_{n,k}=m\}}-p_{n,m}|\le2.
	\end{equation*}
\end{proof}
\subsection{Asymptotic variance}
\begin{lemma}
	\label{lem:bd4var}
	Assume that $np_n^*\to\lambda\ge 0$ and that $n^{r+1}\E\,p_{X_n}^r\to\infty$. Then
	\begin{equation}
	\frac{e^{-2\lambda}}{(r+1)!}\le\liminf_{n\to \infty}\, \frac{\var\,V_{n,r}}{n^{r+1}\,\E\,p_{X_n}^r}\le\limsup_{n\to\infty}\, \frac{\var\,V_{n,r}}{n^{r+1}\,\E\,p_{X_n}^r}\le \frac{1}{r!}.
	\end{equation}
\end{lemma}
\begin{proof}
	Let $p_x=p_{n,x}$, $q_x=1-p_x$ and
	\begin{equation}
	\label{eq:Tnk}
	T_{n,k}(x)=\sum_{i=r-N_{n,k}(x)}^{n-k}\,\binom{n-k}{i}\,p_{x}^i\,q_{x}^{n-k-i}.
	\end{equation}
	Then
	\begin{equation*}
	\begin{split}
	\E\Big(\,\tfrac{I_{\{X_{n,k}=X_n\}}-\pxn}{\pxn}I_{\{N_{n,k}(X_n)+N_{n,k+1}^{n+1}(X_n)\ge r\}}&\Big|X_n,\cF_{n,k}\Big)\\
	&=\E\Big(\,\tfrac{I_{\{X_{n,k}=X_n\}}-\pxn}{\pxn}T_{n,k}(X_n)\,\Big|X_n,\cF_{n,k}\Big)
	\end{split}
	\end{equation*}
	and so
	\begin{equation*}
	d_{n,k}=\E\Big(\,\tfrac{I_{\{X_{n,k}=X_n\}}-\pxn}{\pxn}T_{n,k}(X_n)\,\Big|\cF_{n,k}\Big).
	\end{equation*}
	Also, recalling that $X_n,Y_n, X_{n,1},\ldots,X_{n,n}$ are iid,
	\begin{equation*}
	\begin{split}
	d_{n,k}^2&=\E\Big(\,\tfrac{I_{\{X_{n,k}=X_n\}}-\pxn}{\pxn}T_{n,k}(X_n)\,\Big|\cF_{n,k}\Big)\E\left(\,\tfrac{I_{\{X_{n,k}=Y_n\}}-\pyn}{\pyn}T_{n,k}(Y_n)\,\Big|\cF_{n,k}\right)\\
	&=\E\Big(\,\tfrac{I_{\{X_{n,k}=X_n\}}-\pxn}{\pxn}T_{n,k}(X_n)\,\tfrac{I_{\{X_{n,k}=Y_n\}}-\pyn}{\pyn}T_{n,k}(Y_n)\,\Big|\cF_{n,k}\Big),
	\end{split}
	\end{equation*}
	where the second equality above follows from the conditional independence of $\tfrac{I_{\{X_{n,k}=X_n\}}-\pxn}{\pxn}T_{n,k}(X_n)$ and $\tfrac{I_{\{X_{n,k}=Y_n\}}-\pyn}{\pyn}T_{n,k}(Y_n)$, given $\cF_{n,k}$.
	
	In what follows we compute $\E(d_{n,k}^2|\cF_{n,k-1})$ by considering the cases $X_n=Y_n$ and $X_n\not=Y_n$. We get
	\begin{equation*}
	\begin{split}
	\E(d_{n,k}^2&|\cF_{n,k-1})=\E\Big(\,I_{\{X_n=Y_n\}}\Big(\tfrac{I_{\{X_{n,k}=X_n\}}-\pxn}{\pxn}\Big)^2T_{n,k}^2(X_n)\,\Big|\cF_{n,k-1}\Big)\\
	&+\E\Big(\,I_{\{X_n\not=Y_n\}}\tfrac{\left(I_{\{X_{n,k}=X_n\}}-\pxn\right)\left(I_{\{X_{n,k}=Y_n\}}-\pyn\right)}{\pxn\pyn}T_{n,k}(X_n)T_{n,k}(Y_n)\,\Big|\cF_{n,k-1}\Big)\\
	&=\E\Big(I_{\{X_n=Y_n\}}\tfrac{1-\pxn}{\pxn}T_{n,k}^2(X_n)\Big|\cF_{n,k-1}\Big)-\E\Big(I_{\{X_n\not=Y_n\}}T_{n,k}(X_n)T_{n,k}(Y_n)\Big|\cF_{n,k-1}\Big),
	\end{split}
	\end{equation*}
	where the second equality above follows from conditioning inside both expectations above, with respect to $X_n, Y_n,\cF_{n,k-1}$. Finally, integrating out $Y_n$ in the first expectation, we obtain
	\begin{equation}
	\label{eq:vardnk1}
	\E(d_{n,k}^2|\cF_{n,k-1})=\E\Big(\,\qxn T_{n,k}^2(X_n)\,\Big|\cF_{n,k-1}\Big)-\E\Big(\,I_{\{X_n\not=Y_n\}}T_{n,k}(X_n)T_{n,k}(Y_n)\,\Big|\cF_{n,k-1}\Big)
	\end{equation}
	and, consequently,
	\begin{equation}
	\label{eq:vardnk}
	\var\, d_{n,k}=\E\, d_{n,k}^2=\E\,\qxn T_{n,k}^2(X_n)-\E\,I_{\{X_n\not=Y_n\}}T_{n,k}(X_n)T_{n,k}(Y_n).
	\end{equation}
	For the upper bound of the variance note that $0<T_{n,k}(X_n)\le 1$ and thus \eqref{eq:vardnk} implies
	\begin{equation*}
	\var\,d_{n,k}\le\E d_{n,k}^2\le\E\,T_{n,k}(X_n).
	\end{equation*}
	Also,
	\begin{equation*}
	\E\,( T_{n,k}(X_n)\,|X_n,\cF_{n,k})=\P\,( N_{n,k}(X_n)+N_{n,k+1}^{n+1}(X_n)\ge r\,|X_n,\cF_{n,k})
	\end{equation*}
	and so,
	\begin{equation}
	\label{eq:eTkn}
	\begin{split}
	\E\,( T_{n,k}(X_n)\,|X_n)&=\P\,( N_{n,k}(X_n)+N_{n,k+1}^{n+1}(X_n)\ge r\,|X_n)\\
	&\le \P( N_{n,n+1}(X_n)\ge r\,|X_n).
	\end{split}
	\end{equation}
	Now, recalling that $N_{n,n+1}(m)$ has distribution
	$\Bin(n, p_{n,m})$, for $m\in M_n$, and using \eqref{eq:binineq}, the rhs of \eqref{eq:eTkn} is bounded by $n^rp^r_{X_n}/r!$. Last, taking expectations, we obtain $\var\,d_{n,k}\le n^r\E\, p^r_{X_n}/r!$ and, consequently,
	\begin{equation}
	\label{eq:ub}
	\var\,V_{n,r}= \sum_{k=1}^n\,\var\,d_{n,k}\le \frac{n^{r+1}\E\,p_{X_n}^r}{r!}.
	\end{equation}
	Now, to bound the variance of $d_{n,k}$ from below,  we first find an upper bound for the last term (with minus sign) in display \eqref{eq:vardnk}. To that end note that $T_{n,k}(x), x\in M_m$, defined in \eqref{eq:Tnk}, can be written as
	\begin{equation*}
	T_{n,k}(x)=\P(B_{n-k}(x)+N_{n,k}(x)\ge r|\cF_{n,k}),
	\end{equation*}
	where $B_{n-k}(x)$ is $\Bin(n-k,p_x)$, independent of $X_n,Y_n,\cF_{n,n}$, so
	\begin{equation*}
	T_{n,k}(X_n)=\P(B_{n-k}(X_n)+N_{n,k}(X_n)\ge r|X_n,\cF_{n,k}),
	\end{equation*}
	and
	\begin{equation}
	\label{eq:Tnkx}
	\E(T_{n,k}(x)|X_n)=\P(B_{n-k}(x)+N_{n,k}(x)\ge r|X_n).
	\end{equation}
	Furthermore, for $y\ne x$, let $B_{n-k}(y)$   be     $\Bin(n-k,p_y)$, independent of $X_n,Y_n,\cF_{n,n}$ and independent of $B_{n-k}(x)$. Then
	\begin{equation*}
	J_{n,k}:=\E\left(\,I_{\{X_n\not=Y_n\}}T_{n,k}(X_n)T_{n,k}(Y_n)\,|X_n,Y_n,\cF_{n,k}\right)
	\end{equation*}
	can be written as
	\begin{equation*}
	J_{n,k}=\P\left(X_n\not=Y_n, B_{n,k}(X_n)+N_{n,k}(X_n)\ge r, B_{n,k}(Y_n)+N_{n,k}(Y_n)\ge r\,|X_n,Y_n,\cF_{n,k}\right)
	\end{equation*}
	and so,
	\begin{equation*}
	\E(J_{n,k}|X_n,Y_n)=\P(X_n\not=Y_n, B_{n,k}(X_n)+N_{n,k}(X_n)\ge r, B_{n,k}(Y_n)+N_{n,k}(Y_n)\ge r|X_n,Y_n).
	\end{equation*}
	Then, since, conditionally on $X_n, Y_n$,  $(N_{n,k}(X_n), N_{n,k}(Y_n))$ is   $\mathrm{Mn}_2(k-1,p_{X_n},p_{Y_n})$ and because of the NOD property, we have $\E(J_{n,k}|X_n,Y_n)$
	\begin{equation*}
	\begin{split}
	&\le I_{\{X_n\not=Y_n\}}\P\left(B_{n,k}(X_n)+N_{n,k}(X_n)\ge r|X_n,Y_n)\P(B_{n,k}(Y_n)+N_{n,k}(Y_n)\ge r\,|X_n,Y_n\right)\\
	&= I_{\{X_n\not=Y_n\}}\P\left(B_{n,k}(X_n)+N_{n,k}(X_n)\ge r|X_n)\P(B_{n,k}(Y_n)+N_{n,k}(Y_n)\ge r\,|Y_n\right)\\
	&= I_{\{X_n\not=Y_n\}}\E\left(T_{n,k}(X_n)|X_n)\E(T_{n,k}(Y_n)|Y_n\right)\\
	&= \E(T_{n,k}(X_n)|X_n)\E(T_{n,k}(Y_n)|Y_n)-I_{\{X_n=Y_n\}}(\E(T_{n,k}(X_n)|X_n))^2,
	\end{split}
	\end{equation*}
	where the second equality follows from the NOD property and the third from \eqref{eq:Tnkx}. Finally, taking expectations and using the independence of $X_n$ and $Y_n$, we get
	\begin{equation*}
	\E J_{n,k}\le(\E T_{n,k}(X_n))^2-\E \pxn(\E(T_{n,k}(X_n)|X_n))^2.
	\end{equation*}
	Replacing the rightmost expectation in display \eqref{eq:vardnk} by the bound above we have
	\begin{eqnarray*}
		\var\,d_{n,k}&\ge& \E\,T^2_{n,k}(X_n)-\E\,\pxn T_{n,k}^2(X_n)-\left(\E\,T_{n,k}(X_n)\right)^2+\E\,\pxn(\E\,T_{n,k}(X_n)|X_n)^2\\&=&
		\var\,T_{n,k}(X_n)-\E\,\pxn\var(T_{n,k}(X_n)|X_n)\\&\ge&
		\E\,\var\,(T_{n,k}(X_n)|X_n)-\E\,\pxn\var(T_{n,k}(X_n)|X_n).
	\end{eqnarray*}
	Note that
	\begin{equation*}
	\begin{split}
	\E\,\pxn\var(T_{n,k}(X_n)|X_n)&\le \E\,\pxn\E\,(T^2_{n,k}(X_n)|X_n)\\
	&\le \E\,\pxn T_{n,k}(X_n)
	\le\frac{n^{r}\E\,p_{X_n}^{r+1}}{r!}
	\le \frac{np_n^*}{nr!}\,n^{r}\E\,p_{X_n}^r.
	\end{split}
	\end{equation*}
	Hence, since $np_n^*\to\lambda$,
	\begin{equation*}
	\sum_{k=1}^n\,\var\,d_{n,k}\ge \sum_{k=1}^n\,\E\,\var(T_{n,k}(X_n)|X_n)+o(n^{r+1}\E\,p_{X_n}^r).
	\end{equation*}
	Finally note that $T_{n,k}(x)$, as defined in \eqref{eq:Tnk}, can be written in the form
	\begin{equation*}
	T_{n,k}(x)=\sum_{j=0}^{\infty}P_{j}(x)\,I_j(x),\; x\in M_n,
	\end{equation*}
	where $P_{j}(x)=\binom{n-k}{j}\,p_{x}^j\,q_{x}^{n-k-j}$ and $I_j(x)=I_{\{N_{n,k}(x)\ge r-j\}}$.
	Therefore,
	\begin{equation*}
	\var\,T_{n,k}(x)=\sum_{j=0}^{\infty}P_j^2(x)\var\,I_j(x)+2\sum_{j_1<j_2}P_{j_1}(x)P_{j_2}(x)\cov\left(I_{j_1}(x),I_{j_2}(x)\right).
	\end{equation*}
	Since $I_{j_1}(x)\le I_{j_2}(x)$ it follows that the double sum above is non-negative and so,
	\begin{equation*}
	\begin{split}
	\var\,T_{n,k}(x)&\ge\sum_{j=0}^{\infty}P_j^2(x)\var\,I_j(x)\\
	&\ge P_0^2(x)\var\,I_0(x)\\
	&=(1-p_x)^{2(n-k)}\P(N_{n,k}(x)\ge r)\P(N_{n,k}(x)< r)\\
	&\ge (1-p_x)^{2(n-k)}\P(N_{n,k}(x)= r)\P(N_{n,k}(x)=0)\\
	&=\binom{k-1}{r}p_x^r(1-p_x)^{2n-r-2}\ge \binom{k-1}{r}p_x^r(1-p_n^*)^{2n}.
	\end{split}
	\end{equation*}
	Consequently,
	\begin{equation*}
	\var (T_{n,k}(X_n)|X_n)\ge \binom{k-1}{r}p^{r}_{X_n}(1-p_n^*)^{2n}
	\end{equation*}
	and finally, since $np_n^*\to\lambda$,
	\begin{equation*}
	\sum_{k=1}^n\,\var\,d_{n,k}\ge (1-p_n^*)^{2n}\,\E\,p_{X_n}^r \sum_{k=1}^n\,\binom{k-1}{r}+o(n^{r+1}\E\,p_{X_n}^r)\sim \tfrac{e^{-2\lambda}}{(r+1)!}\,n^{r+1}\E\,p_{X_n}^r.
	\end{equation*}
\end{proof}
\subsection{Variance of the sum of conditional variances}
\begin{lemma}
	\label{lem:vcv}
	Under the hypotheses of Lemma \ref{lem:bd4var}
	\begin{equation}
	\label{eq:varconvar}
	W_n:=\var\,\sum_{k=1}^n\,\E\,(d_{n,k}^2|\cF_{n,k-1})=o((n^{r+1}\E\,p_{X_n}^r)^2).
	\end{equation}
\end{lemma}
\begin{proof}
	We first rewrite \eqref{eq:vardnk1} as
	\begin{equation}
	\label{eq:ednk2}
	\E\,(d_{n,k}^2|\cF_{n,k-1})=\E(A_{n,k}(X_n)-B_{n,k}(X_n,Y_n)|\cF_{n,k-1})=\alpha_{n,k}-\beta_{n,k},
	\end{equation}
	where
	\begin{equation*}
	\begin{split}
	&A_{n,k}(x)=q_{x}T_{n,k}^2(x),\quad \alpha_{n,k}=\E(A_{n,k}(X_n)|\cF_{n,k-1}),\\ &B_{n,k}(x,y)=I_{\{x\neq y\}}T_{n,k}(x)T_{n,k}(y),\quad \beta_{n,k}=\E(B_{n,k}(X_n,Y_n)|\cF_{n,k-1}).
	\end{split}
	\end{equation*}
	Consequently, letting $W_n^\alpha=\var\sum_{k=1}^n\,\alpha_{n,k}$, $W_n^\beta=\var\sum_{k=1}^n\,\beta_{n,k}$ and noting that $\var\, (X+Y)\le 2(\var\, X+\var\, Y)$, we have
	\begin{equation}
	\label{eq:Wn}
	W_n\le 2W_n^\alpha+2W_n^\beta.
	\end{equation}
	Then
	\begin{equation}
	\label{eq:Wnalpha}
	W_n^\alpha=\sum_{k=1}^n\,\var\,\alpha_{n,k}+2\sum_{1\le k<l\le n}\,\cov\,(\alpha_{n,k},\alpha_{n,l}),
	\end{equation}
	and the analogous formula holds for $W_n^\beta$. In what follows we express the variances and covariances of $\alpha_{n,k},\beta_{n,k}$ in terms of $A_{n,k}(X_n), B_{n,k}(X_n,Y_n)$. For simplicity, let $Z_n=(X_n,Y_n), Z'_n=(X'_n,Y'_n)$, then
	\begin{equation}
	\label{eq:covars}
	\begin{split}
	\var\, \alpha_{n,k}&=\cov(A_{n,k}(X_n),A_{n,k}(X_n')), \cov\,(\alpha_{n,k},\alpha_{n,l})=\cov(A_{n,k}(X_n),A_{n,l}(X_n')),\\
	\var\, \beta_{n,k}&=\cov(B_{n,k}(Z_n),B_{n,k}(Z_n')), \cov\,(\beta_{n,k},\beta_{n,l})=\cov(B_{n,k}(Z_n),B_{n,l}(Z_n')),
	\end{split}
	\end{equation}
	where $X_n'$ and $Y_n'$ are  such that $X_n,X_n',Y_n,Y_n',X_{n,1},\ldots,X_{n,n}$ are iid for any $n\ge 1$. We only check the first formula; the others are obtained similarly.
	\begin{equation*}
	\begin{split}
	\E\,\alpha_{n,k}^2&=\E\,(\E(A_{n,k}(X_n)|\cF_{n,k-1})\E(A_{n,k}(X_n')|\cF_{n,k-1}))\\
	&=\E\,(\E(A_{n,k}(X_n)A_{n,k}(X_n')|\cF_{n,k-1}))\\
	&=\E\,(A_{n,k}(X_n)A_{n,k}(X_n')),
	\end{split}
	\end{equation*}
	\begin{equation*}
	(\E\,\alpha_{n,k})^2=(\E\,A_{n,k}(X_n))^2=(\E\,A_{n,k}(X_n))(\E\,A_{n,k}(X_n')),
	\end{equation*}
	and the formula for $\var\, \alpha_{n,k}$ follows. We now compute bounds for the covariances in \eqref{eq:covars}. Since $A_{n,k}(x)$ and $B_{n,k}(x,y)$ are bounded above by $T_{n,k}(x)\le 1$ 
	reasoning as in the paragraph preceding \eqref{eq:ub}, we have,
	\begin{equation}
	\label{eq:covAnkAnk}
	\cov(A_{n,k}(X_n),A_{n,k}(X_n'))\le \E\,A_{n,k}(X_n)A_{n,k}(X_n')\le \E\,T_{n,k}(X_n)\le n^{r}\E\,p_{X_n}^r
	\end{equation}
	and
	\begin{equation}
	\label{eq:covBnkBnk}
	\cov(B_{n,k}(Z_n),B_{n,k}(Z_n'))\le \E\,B_{n,k}(Z_n)B_{n,k}(Z_n')\le \E\,T_{n,k}(X_n)\le n^{r}\E\,p_{X_n}^r.
	\end{equation}
	Next, we handle $\cov(A_{n,k}(X_n),A_{n,l}(X_n'))$, which requires somewhat more effort than the previous covariances because the crude bounds do not yield the right order in $n$. Since $A_{n,k}(x)=(1-p_x)T^2_{n,k}(x)$,
	\begin{equation}
	\label{eq:rem3}
	\cov\,(A_{n,k}(X_n),A_{n,l}(X'_n))=\cov\,(T_{n,k}^2(X_n),T^2_{n,l}(X'_n))+O(n^r\E\,p_{X_n}^{r+1})
	\end{equation}
	because each of the remaining three covariances is bounded by an expression of the form $\E p_{X_n}T_{n,k}(X_n)\le cn^r\E p_{X_n}^{r+1}$. To bound the covariance between $T_{n,k}^2(X_n)$ and $T^2_{n,l}(X'_n)$
	we write
	\begin{equation}
	\label{eq:ET2T2}
	\E T^2_{n,k}(X_n) T^2_{n,l}(X'_n)=\E I_{\{X_n=X'_n\}}T^2_{n,k}(X_n) T^2_{n,l}(X_n)+
	\E I_{\{X_n\ne X'_n\}}T^2_{n,k}(X_n) T^2_{n,l}(X'_n).
	\end{equation}
	and note that the first expectation in \eqref{eq:ET2T2} is bounded by
	\begin{equation}
	\label{eq:eqbound}
	\E I_{\{X_n=X'_n\}}T_{n,k}(X_n)=\E p_{X_n}T_{n,k}(X_n)\le cn^r\E p_{X_n}^{r+1},
	\end{equation}
	where $c$ is a positive constant. For  the second expectation in \eqref{eq:ET2T2} we have the following expression,  written in terms of (conditionally independent) binomial random variables $B_1, B_2, B_1', B_2'$.
	\begin{equation}
	\label{eq:binomials}
	\begin{split}
	\E I_{\{X_n\ne X'_n\}}\P(B_{1}&\ge r-N_{n,k}(X_n),B_{2}\ge r-N_{n,k}(X_n),\\
	&B'_{1}\ge r-N_{n,l}(X'_n),B'_{2}\ge r-N_{n,l}(X'_n)|X_n,X'_n).
	\end{split}
	\end{equation}
	Conditionally on $(X_n,X_n')$, $B_1, B_2, B_1', B_2'$ are independent, with $B_1, B_2$ distributed $\Bin(n-k,p_{X_n})$ and $B'_1, B'_2$ distributed $\Bin(n-k,p_{X'_n})$. Further, $B_1, B_2, B_1', B_2'$ are independent of $\cF_{n,k}, \cF_{n,l}$, conditionally on $(X_n,X_n')$.
	
	Note that \eqref{eq:binomials} can be rewritten as
	\begin{equation}
	\label{eq:minbinomials}
	\E I_{\{X_n\ne X'_n\}}\P(N_{n,k}(X_n)\ge r-B_{12}, N_{n,l}(X'_n)\ge r-B'_{12}\}|X_n,X'_n),
	\end{equation}
	where $B_{12}=\min\{B_1,B_2\}$ and $B'_{12}=\min\{B'_1,B'_2\}$. Note also that, for $x\ne y$, $N_{n,k}(x)$ and $N_{n,l}(y)$ are NOD; see \eqref{eq:covH}. Thus, conditioning on the values of the binomials,  using the NOD property; then  integrating over the $B$'s and using independence of $X_n$ and $X'_n$, we have the following upper bound for \eqref{eq:minbinomials}
	\begin{equation*}
	\E I_{\{X_n\ne X'_n\}}\P(N_{n,k}(X_n)\ge r-B_{12}|X_n)\P( N_{n,l}(X'_n)\ge r-B'_{12}\}|X'_n),
	\end{equation*}
	which, after ignoring the indicator and noting that the conditional probabilities (on $X_n$ and $X_n'$) are independent random variables, can be finally bounded by
	\begin{equation}
	\label{eq:finalbound}
	\E \P(N_{n,k}(X_n)\ge r-B_{12}|X_n)\E\P( N_{n,l}(X'_n)\ge r-B'_{12}\}|X'_n)=\E\,T_{n,k}^2(X_n)\E\,T_{n,l}^2(X'_n).
	\end{equation}
	Therefore, from \eqref{eq:rem3}, \eqref{eq:ET2T2}, \eqref{eq:eqbound} and \eqref{eq:finalbound}, we have
	\begin{equation}
	\label{eq:covT2nk}
	\cov\,(T^2_{n,k}(X_n),T^2_{n,l}(X'_n))\le cn^r\E\,p_{X_n}^{r+1}.
	\end{equation}
	It remains to bound the covariances $\cov\,(B_{n,k}(Z_n),B_{n,l}(Z'_n))$. To that end we consider first, the expected value of the product.
	\begin{equation}
	\label{eq:bnkbnl}
	\begin{split}
	\E\, B_{n,k}(Z_n)B_{n,l}(Z'_n)&\le \E\, I_DT_{n,k}(X_n)T_{n,k}(Y_n)T_{n,l}(X'_n)T_{n,l}(Y'_n)\\
	&\quad+\E\, I_{D^c}T_{n,k}(X_n)T_{n,k}(Y_n)T_{n,l}(X'_n)T_{n,l}(Y'_n),
	\end{split}
	\end{equation}
	where $D$ is the event that 
	$X_n,Y_n,X'_n,Y'_n$ are all distinct. Then,
	\begin{equation}
	\label{eq:dc}
	\E I_{D^c}T_{n,k}(X_n)T_{n,k}(Y_n)T_{n,l}(X'_n)T_{n,l}(Y'_n)\le {4\choose2}\E p_{X_n}T_{n,k}(X_n)\le cn^r\E p_{X_n}^{r+1}.
	\end{equation}
	Note that, as in \eqref{eq:binomials}, the first term on the rhs of \eqref{eq:bnkbnl} can be written as follows
	\begin{equation}
	\label{eq:binomials2}
	\begin{split}
	\E I_D\P(B_{1}\ge r-N_{n,k}(X_n)&,B_{2}\ge r-N_{n,k}(Y_n),\\
	&B'_{1}\ge r-N_{n,l}(X'_n),B'_{2}\ge r-N_{n,l}(Y'_n)|Z_n,Z'_n).
	\end{split}
	\end{equation}
	Conditionally on $(Z_n,Z'_n)$, $B_1, B_2, B_1', B_2'$ are independent, where $B_1$ is $\Bin(n-k,p_{X_n})$, $B_2$ is $\Bin(n-k,p_{Y_n})$,  $B'_1$ is $\Bin(n-k,p_{X'_n})$ and $B'_2$ is $\Bin(n-k,p_{Y'_n})$ . Also, $B_1, B_2, B_1', B_2'$ are independent of $\cF_{n,k}, \cF_{n,l}$, conditionally on $(Z_n, Z_n')$. Now, using the NOD property \eqref{eq:PA2} and the independence of $X_n,Y_n,X'_n$, $Y'_n$, the expression in \eqref{eq:binomials2} is bounded above by
	\begin{equation}
	\label{eq:binomials3}
	\begin{split}
	\E \P(B_{1}\ge r-N_{n,k}(X_n)&, B_{2}\ge r-N_{n,k}(Y_n)|Z_n)\\
	&\quad\times\P(B'_{1}\ge r-N_{n,l}(X'_n), B'_{2}\ge r-N_{n,l}(Y'_n)|Z'_n)\\
	&=\E \P(B_{1}\ge r-N_{n,k}(X_n), B_{2}\ge r-N_{n,k}(Y_n)|Z_n)\\
	&\quad\times\E\P(B'_{1}\ge r-N_{n,l}(X'_n), B'_{2}\ge r-N_{n,l}(Y'_n)|Z'_n)\\
	&=\E \,T_{n,k}(X_n)T_{n,k}(Y_n)\E\, T_{n,l}(X'_n)T_{n,l}(Y'_n)\\
	&=\E \, B_{n,k}(Z_n)\E\, B_{n,l}(Z'_n)+O(n^r\E\,p_{X_n}^{r+1}).
	\end{split}
	\end{equation}
	Therefore, from \eqref{eq:bnkbnl}, \eqref{eq:dc} and \eqref{eq:binomials3},
	\begin{equation}
	\label{eq:covBnkBnl}
	\cov\,(B_{n,k}(Z_n),B_{n,l}(Z'_n))\le cn^r\E p_{X_n}^{r+1}.
	\end{equation}
	We complete
	the proof of \eqref{eq:varconvar} by collecting the partial results above to obtain bounds for $W_n^\alpha$ and $W_n^\beta$, using formula \eqref{eq:Wnalpha}. From \eqref{eq:covAnkAnk} and \eqref{eq:covBnkBnk} we have
	\begin{equation*}
	\label{eq:sumvaralphabeta}
	\sum_{k=1}^n\,\var\,\alpha_{n,k}\le n^{r+1}\E\,p_{X_n}^r\quad\text{and}\quad \sum_{k=1}^n\,\var\,\beta_{n,k}\le n^{r+1}\E\,p_{X_n}^r
	\end{equation*}
	From \eqref{eq:rem3} and \eqref{eq:covT2nk}
	\begin{equation*}
	\label{eq:sumcovalpha}
	\sum_{1\le k<l\le n}\cov\,(\alpha_{n,k},\alpha_{n,l})\le c{n\choose2}n^r\E p_{X_n}^{r+1}\le cnp_n^*\left(n^{r+1}\E p_{X_n}^r\right)=o((n^{r+1}\E p_{X_n}^r)^2).
	\end{equation*}
	Last, from \eqref{eq:covBnkBnl}
	\begin{equation*}
	\label{eq:sumcovbeta}
	\sum_{1\le k<l\le n}\cov\,(\beta_{n,k},\beta_{n,l})\le c{n\choose 2}n^r\E\,p_{X_n}^{r+1}\le cnp_n^*\left(n^{r+1}\E p_{X_n}^r\right) =o((n^{r+1}\E p_{X_n}^r)^2).
	\end{equation*}
	The conclusion follows from \eqref{eq:Wn}, \eqref{eq:Wnalpha} and the bounds for the sums of variances and covariances above.
\end{proof}
\subsection{Final touch - the martingale CLT}
\label{sec:finaltouch}
We show the asymptotic normality by applying the martingale central limit theorem (see e.~g. \cite[Theorem~2.5]{HI} to the martingale differences $(d_{n,k})$. Since $d_{n,k}$'s are uniformly bounded the conditional Lindeberg condition (\cite[condition (2.5)]{HI}) follows from the fact that the variance of the sum  grows to infinity as $n\to\infty$. The remaining condition to be checked (\cite[condition (2.7)]{HI}) is that
\begin{equation*}
\label{lind}
\frac{\sum_{k=1}^n\E (d_{n,k}^2|\cF_{n,k-1})}{\sum_{k=1}^n\E d_{n,k}^2}\stackrel \P\longrightarrow1,
\end{equation*}
as $n\to\infty$ or, equivalently,   that
\begin{equation*}
\frac{\sum_{k=1}^n(\E (d^2_{n,k}|\cF_{n,k-1})-\E d_{n,k}^2)}{\sum_{k=1}^n\E d^2_{n,k}}\stackrel \P\longrightarrow 0.
\end{equation*}
But this follows immediately from Lemma \ref{lem:bd4var}, Lemma \ref{lem:vcv} and Chebyshev's inequality.
\section{Asymptotics for  number of full containers with and without overflow}
Let $L_{n,r}$ denote  the number of full containers and $M_{n,r}$ denote number of full containers without overflow. The main idea is to represent $L_{n,r}$ and $M_{n,r}$ in terms of  the size of the overflow $V_{n,r}$.

Recall that $N_{n,n+1}(m)$ is the total number of balls in the sample for which the $m$th box  was selected. Thus
\begin{equation*}
L_{n,r}=\sum_{j\in M}\,I_{\{N_{n,n+1}(j)\ge r\}}\quad\text{and}\quad M_{n,r}=\sum_{j\in M}\,I_{\{N_{n,n+1}(j)=r\}}=L_{n,r}-L_{n,r+1}.
\end{equation*}
We note that
\begin{equation*}
\begin{split}
L_{n,r}&=\sum_{j\in M}\,\sum_{k=2}^n\,I_{\{X_{n,k}=j\}}\,I_{\{N_{n,k-1}(j)=r-1\}}\\
&=\sum_{j\in M}\,\sum_{k=2}^n\,I_{\{X_{n,k}=j\}}\,I_{\{N_{n,k-1}(j)\ge r-1\}}-\sum_{j\in M}\,\sum_{k=2}^n\,I_{\{X_{n,k}=j\}}\,I_{\{N_{n,k-1}(j)\ge r\}}.
\end{split}
\end{equation*}
That is,
\begin{equation}
\label{eq:Lnr}
L_{n,r}=V_{n,r-1}-V_{n,r}
\end{equation}
and
\begin{equation}
\label{eq:Mnr}
M_{n,r}=V_{n,r-1}-2V_{n,r}+V_{n,r+1}.
\end{equation}
Note that in the case $r=1$ we have $V_{n,0}=n$ and thus $L_{n,1}$, which is a number of non-empty boxes, is
\begin{equation}
\label{eq:Ln1}
L_{n,1}=n-V_{n,1}
\end{equation}
and $M_{n,1}$, which is number of singleton boxes, is
\begin{equation}
\label{eq:Mn1}
M_{n,1}=n-2V_{n,1}+V_{n,2}.
\end{equation}
These representations of $M_{n,r}$ and $L_{n,r}$  in terms of $V_{n,r-1}$, $V_{n,r}$ and $V_{n,r+1}$ allow to read Poissonian asymptotics of these two sequences from Theorem \ref{thm:rrr}. For $M_{n,r}$  the forthcoming statement  was proved in \cite[Theorem~III.3.1]{KSC}.
\begin{theorem}
	\label{thm:ml}
	Assume that $np_n^*\to 0$.
	\begin{enumerate}
		\item 	If $r>1$ and  $n^r\,\E\,p_{X_n}^{r-1}\to r!\mu$ then
		\begin{equation*}
		L_{n,r}\stackrel{d}{\to}\,\mathrm{Pois}(\mu)\quad \mbox{and}\quad M_{n,r}\stackrel{d}{\to}\,\mathrm{Pois}(\mu).
		\end{equation*}
		\item If $r=1$ and  $n^2\,\E\,\pxn\to \mu$ then
		\begin{equation*}
		n-L_{n,1}\stackrel{d}{\to}\,\mathrm{Pois}(\mu)\quad \mbox{and}\quad \frac{n-M_{n,1}}{2}\stackrel{d}{\to}\mathrm{Pois}(\mu).
		\end{equation*}
	\end{enumerate}
\end{theorem}
\begin{proof}
	The case $r>1$: Due to representations \eqref{eq:Lnr} and \eqref{eq:Mnr} to prove both results it suffices to show that $\E\,V_{n,s}\to 0$ for any fixed $s\ge r$. But following the argument from the beginning of Step II of the proof of Theorem \ref{thm:rrr} we see that
	\begin{equation*}
	\E\,V_{n,s}=\sum_{k=1}^n\,\sum_{i=s}^{k-1}\,{k-1 \choose i}\,\E\,p_{X_n}^iq_{X_n}^{k-1-i}\le n^{s+1}\E\,p_{X_n}^s
	\to 0,
	\end{equation*}
	where the convergence to zero in the last step follows from  Lemma~\ref{lem:to_zero}.
	
	The case $r=1$: The first part follows from Theorem \ref{thm:rrr} since  \eqref{eq:Ln1} implies $n-L_{n,1}=V_{n,1}$. The second follows also from Theorem \ref{thm:rrr} since \eqref{eq:Mn1} gives
	\begin{equation*}
	\frac{n-M_{n,1}}{2}=V_{n,1}-\frac{V_{n,2}}{2}
	\end{equation*}
	and, similarly as in the case $r>1$, we have $\E\,V_{n,2}\to 0$.
\end{proof}
Note that under assumptions of Th. \ref{thm:ml}
\begin{itemize}
	\item in case 1: \hspace{5mm}  $L_{n,r}-M_{n,r}\stackrel{\P}\to 0$,
	\item in case 2: \hspace{5mm} $\tfrac{L_{n,1}-M_{n,1}}{n}\stackrel{\P}{\to}1.$ \end{itemize}

Representations \eqref{eq:Lnr} and \eqref{eq:Mnr} are also useful for getting Gaussian asymptotics of $L_{n,r}$ and $M_{n,r}$ from Theorem \ref{norasy} in the case $\lambda=0$.
\begin{theorem}
	\label{thm:norasy_2}
	Assume that $np_n^*\to 0$ and $r\ge 1$.
	\begin{enumerate}
		\item If $n^{r+1}\E\,p_{X_n}^r\to\infty$ then
		\begin{equation*}
		\frac{L_{n,r}-\E\,L_{n,r}}{\sqrt{\var\,L_{n,r}}}\stackrel{d}{\to} \mathrm{N}(0,1).
		\end{equation*}
		\item If $n^{r+2}\E\,p_{X_n}^{r+1}\to\infty$ then
		\begin{equation*}
		\frac{M_{n,r}-\E\,M_{n,r}}{\sqrt{\var\,M_{n,r}}}\stackrel{d}{\to} \mathrm{N}(0,1).
		\end{equation*}
		
	\end{enumerate}
\end{theorem}
\begin{proof}
	By representation \eqref{eq:Lnr} we can write
	\begin{equation*}
	\tfrac{\var\,L_{n,r}}{\var\,V_{n,r-1}}=1+\tfrac{\var\,V_{n,r}}{\var\,V_{n,r-1}}-2\tfrac{\cov(V_{n,r-1},\,V_{n,r})}{\var\,V_{n,r-1}}.
	\end{equation*}
	Since $n^{r+1}\E\,p_{X_n}^r\le np_n^*\,n^r\E\,p_{X_n}^{r-1}$ it follows that $n^r\E\,p_{X_n}^{r-1}\to \infty$. Therefore by Lemma ~\ref{lem:bd4var} we have
	\begin{equation*}
	\tfrac{\var\,V_{n,r}}{\var\,V_{n,r-1}}\le c \tfrac{n^{r+1}\E\,p_{X_n}^r}{n^r\E\,p_{X_n}^{r-1}}\le c np_n^*\to 0
	\end{equation*}
	and thus also
	\begin{equation*}
	\left|\tfrac{\cov(V_{n,r-1},\,V_{n,r})}{\var\,V_{n,r-1}}\right|\le\sqrt{\tfrac{\var\,V_{n,r}}{\var\,V_{n,r-1}}}\to 0.
	\end{equation*}
	Consequently,  $\tfrac{V_{n,r}-\E\,V_{n,r}}{\sqrt{\var\,L_{n,r}}}\stackrel{L^2}{\to}0$.  Thus the first result is a consequence of Theorem~\ref{norasy} since, in view of the representation \eqref{eq:Lnr},
	\begin{equation*}
	\tfrac{L_{n,r}-\E\,L_{n,r}}{\sqrt{\var\,L_{n,r}}}=\tfrac{V_{n,r-1}-\E\,V_{n,r-1}}{\sqrt{\var\,V_{n,r-1}}}\,\sqrt{\tfrac{\var\,V_{n,r-1}}{\var\,L_{n,r}}}-\tfrac{V_{n,r}-\E\,V_{n,r}}{\sqrt{\var\,L_{n,r}}}.
	\end{equation*}
	For the second case, by representation \eqref{eq:Mnr} we can write
	\begin{equation*}
	\begin{split}
	\tfrac{\var\,M_{n,r}}{\var\,V_{n,r-1}}&=1+4\tfrac{\var\,V_{n,r}}{\var\,V_{n,r-1}}+\tfrac{\var\,V_{n,r+1}}{\var\,V_{n,r-1}}\\
	&-4\tfrac{\cov(V_{n,r-1},\,V_{n,r})}{\var\,V_{n,r-1}}
	-4\tfrac{\cov(V_{n,r},\,V_{n,r+1})}{\var\,V_{n,r-1}}+2\tfrac{\cov(V_{n,r-1},\,V_{n,r+1})}{\var\,V_{n,r-1}}.
	\end{split}
	\end{equation*}
	Similarly as in the previous case we conclude that $n^s\E\,p_{X_n}^{s-1}\to \infty$ for $s=r,r+1$. Therefore, by  the same argument as above it follows that each of the summands at the right hand side above except the first one converges to 0 as $n\to \infty$.
	Consequently,  $\tfrac{V_{n,s}-\E\,V_{n,s}}{\sqrt{\var\,M_{n,r}}}\stackrel{L^2}{\to}0$, $s=r,r+1$.  Thus the second result is a consequence of Theorem~\ref{norasy} since, in view of  \eqref{eq:Mnr},
	\begin{equation*}
	\tfrac{M_{n,r}-\E\,M_{n,r}}{\sqrt{\var\,M_{n,r}}}=\tfrac{V_{n,r-1}-\E\,V_{n,r-1}}{\sqrt{\var\,V_{n,r-1}}}\,\sqrt{\tfrac{\var\,V_{n,r-1}}{\var\,M_{n,r}}}-2\tfrac{V_{n,r}-\E\,V_{n,r}}{\sqrt{\var\,M_{n,r}}}+\tfrac{V_{n,r+1}-\E\,V_{n,r+1}}{\sqrt{\var\,M_{n,r}}}.
	\end{equation*}
\end{proof}

\end{document}